\documentclass[a4paper]{amsart}
\usepackage[all]{xy}
\usepackage{amssymb,mathrsfs,enumitem,amsmath,bbold}

\usepackage[in]{fullpage}
\usepackage{color}
\definecolor{Chocolat}{rgb}{0.36, 0.2, 0.09}
\definecolor{BleuTresFonce}{rgb}{0.215, 0.215, 0.36}
\definecolor{EgyptianBlue}{rgb}{0.06, 0.2, 0.65}

\usepackage[colorlinks,final,hyperindex]{hyperref}
\hypersetup{citecolor=BleuTresFonce, urlcolor=EgyptianBlue, linkcolor=Chocolat}

\usepackage{fourier}
\newcommand{\llbincomb}[7]{\ensuremath{\!\!\!\!
 \vcenter{\hbox{\xymatrix@R=.4pc@C=.2pc{ %@R=.7mm@C=1mm{
 			#4\ar@{-}[dr] && #5\ar@{-}[dl]&&\\
 			&*+[o][F-]{#3}\ar@{-}[dr]&&#6\ar@{-}[dl] \\
 			&&*+[o][F-]{#2}\ar@{-}[dr] &   &#7\ar@{-}[dl]\\
 			&&&*+[o][F-]{#1}\ar@{-}[d]&\\
 			&&&*{}&
 		}}}}}

\newcommand{\lrbincomb}[7]{\!\!\!\ensuremath{
 \vcenter{\hbox{\xymatrix@R=.4pc@C=.2pc{ %@R=.7mm@C=1mm{
 			&#5\ar@{-}[dr] && #6\ar@{-}[dl]&\\
 			#4\ar@{-}[dr]&&*+[o][F-]{#3}\ar@{-}[dl] \\
 			&*+[o][F-]{#2}\ar@{-}[dr] &   &#7\ar@{-}[dl]\\
 			&&*+[o][F-]{#1}\ar@{-}[d]&\\
 			&&*{}&
 		}}}}\!\!\!\!}

\newcommand{\rlbincomb}[7]{\ensuremath{
\vcenter{\hbox{\xymatrix@R=.4pc@C=.2pc{ %@R=.7mm@C=1mm{
			#5\ar@{-}[dr] && #6\ar@{-}[dl]&\\
			&*+[o][F-]{#3}\ar@{-}[dr]& & #7\ar@{-}[dl]\\
           #4\ar@{-}[dr] &   &*+[o][F-]{#2}\ar@{-}[dl]&\\
			&*+[o][F-]{#1}\ar@{-}[d]&&\\
			&*{}&&&
}}}}\!\!\!\!\!\!\!}
				
\newcommand{\rrbincomb}[7]{\ensuremath{
\vcenter{\hbox{\xymatrix@R=.4pc@C=.2pc{ %@R=.7mm@C=1mm{
			&&#6\ar@{-}[dr] && #7\ar@{-}[dl]\\
			&#5\ar@{-}[dr]& & *+[o][F-]{#3}\ar@{-}[dl]&\\
			#4\ar@{-}[dr] &   &*+[o][F-]{#2}\ar@{-}[dl]&&\\
			&*+[o][F-]{#1}\ar@{-}[d]&&&\\
			&*{}&&&&
}}}}\!\!\!\!\!\!\!\!\!\!}

\catcode`\@=11

\catcode`\@=13

\newtheorem{theorem}{Theorem}%[subsection]

\newtheorem{lemma}[theorem]{Lemma}
\newtheorem{proposition}[theorem]{Proposition}

\theoremstyle{definition}

\newtheorem{remark}[theorem]{Remark}
\newtheorem{definition}[theorem]{Definition}

\DeclareMathAlphabet{\pazocal}{OMS}{zplm}{m}{n}

\def\calE{\pazocal{E}}

\def\calO{\pazocal{O}}
\def\calP{\pazocal{P}}
\def\calQ{\pazocal{Q}}
\def\calR{\pazocal{R}}
\def\calS{\pazocal{S}}
\def\calT{\pazocal{T}}

\def\calX{\pazocal{X}}
\def\calY{\pazocal{Y}}

\DeclareMathOperator{\id}{id}

\DeclareMathOperator{\Ass}{\textsl{Ass}}
\DeclareMathOperator{\Poisson}{\textsl{Poisson}}
\DeclareMathOperator{\Com}{\textsl{Com}}

\DeclareMathOperator{\PL}{\textsl{PreLie}}

\DeclareMathOperator{\FM}{\textsl{FMan}}
\DeclareMathOperator{\Lie}{\textsl{Lie}}

\DeclareMathOperator{\gr}{gr}

\DeclareMathAlphabet{\mathbbold}{U}{bbold}{m}{n}

\def\k{\mathbbold{k}}

\newcommand{\ac}{\scriptstyle \text{\rm !`}}

\begin{document}

\title{Algebraic structures of $F$-manifolds via pre-Lie algebras}

\author{Vladimir Dotsenko}
\address{School of Mathematics, Trinity College, Dublin 2, Ireland, and Departamento de Matem\'aticas, CINVESTAV-IPN, Av. Instituto Polit\'ecnico Nacional 2508, Col. San Pedro Zacatenco, M\'exico, D.F., CP 07360, Mexico}
\email{vdots@maths.tcd.ie}

\dedicatory{To Yuri Ivanovich Manin,\\ with my deepest respect and admiration}


\begin{abstract}
We relate the operad $\FM$ controlling the algebraic structure on the tangent sheaf of an $F$-manifold (weak Frobenius manifold) defined by Hertling and Manin to the operad $\PL$ of pre-Lie algebras: for the filtration of $\PL$ by powers of the ideal generated by the Lie bracket, the associated graded object is~$\FM$. 
\end{abstract}

\maketitle

\section*{Introduction} The notion of an $F$-manifold (weak Frobenius manifold) was introduced by Hertling and Manin \cite{HeMa} as a relaxation of the notion of a Frobenius manifold. By definition, an $F$-manifold is a pair $(M,-\circ-)$ consisting of a smooth supermanifold $M$ and a smooth bilinear commutative associative product $-\circ-$ on the tangent sheaf~$TM$ satisfying the condition 
 \[
P_{X_1\circ X_2}(X_3,X_4)=X_1\circ P_{X_2}(X_3,X_4) + (-1)^{|X_1||X_2|}X_2\circ P_{X_1}(X_3,X_4) ,  
 \]
where $P_{X_1}(X_2,X_3)=[X_1,X_2\circ X_3]-[X_1,X_2]\circ X_3-(-1)^{|X_1||X_2|}X_2\circ[X_1,X_3]$ measures to what extent the product~$-\circ-$ and the usual Lie bracket of vector fields fail the Poisson algebra axioms. Any Frobenius manifold is an $F$-manifold; any $F$-manifold for which the product $-\circ-$ is semisimple can be made into a Frobenius manifold~\cite{HeMa}. A very deep operadic result on $F$-manifolds was established by Merkulov~\cite{Merk03}: any homotopy Gerstenhaber algebra gives rise to an $F$-manifold (in fact to a much richer structure including higher operations). However, a very fundamental problem of a reasonable description of the operad $\FM$ encoding ``$F$-manifold algebras'', that is algebras with an associative commutative product and a Lie bracket satisfying the Hertling--Manin condition, has never been addressed. On the one hand, it is not surprising, since the Hertling--Manin condition is a cubic relation, and as such the operad $\FM$ is completely outside the scope of applicability of traditional methods of the operad theory like the Koszul duality theory, so there are no readily available methods to study it, as remarked by Manin in \cite[Sec.~3.3, Remark]{Manin17} where the question of studying the operad $\FM$ is emphasised. On the other hand, this algebraic structure appears as fundamental as the one controlled by the Poisson operad, and its satisfactory algebraic description is highly desirable.

In this paper, we give a very direct description of the operad $\FM$. Quite surprisingly, this operad turns out to be intimately related to one of the most famous operads appearing in the literature, the operad $\PL$ of pre-Lie algebras \cite{CL} that controls bilinear products $-\cdot-$ on a graded vector space $V$ satisfying the condition
\begin{equation}\label{eq:pre-Lie}
(X_1\cdot X_2)\cdot X_3 - X_1\cdot (X_2\cdot X_3) = (-1)^{|X_2||X_3|}((X_1\cdot X_3)\cdot X_2- X_1\cdot (X_3\cdot X_2)) .
\end{equation}
This condition implies that the bracket $[X_1,X_2]=X_1\cdot X_2-(-1)^{|X_1||X_2|}X_2\cdot X_1$ satisfies the Jacobi identity. 
The slogan promoted by this note is ``$F$-manifold algebras are the same to pre-Lie algebras as Poisson algebras to associative algebras.'' (The word ``algebras'' is used in that slogan four times to emphasize that we consider \emph{algebraic} properties of the structure on the tangent sheaf of an $F$-manifold, and ignore the geometry of the arguments, the vector fields.) 

In fact, there exists a commutative diagram of operads
 \[
 \xymatrix{
\PL\ar@{->>}[rr] \ar@{->}[d] & & \Ass  \ar@{->}[d] \\ 
 \FM\ar@{->>}[rr]  & &\Poisson      
 }
 \]
where the vertical arrows are isomorphisms on the level of $\mathbb{S}$-modules (but not on the level of operads). Note that in contrast with the operad $\FM$ the defining relations of the operad $\PL$ are quadratic, and the symmetrised pre-Lie product $X_1\circ X_2=X_1\cdot X_2+(-1)^{|X_1||X_2|}X_2\cdot X_1$ does not satisfy any identities at all, as established in~\cite{BL}.  Our approach substantially relies on shuffle operads and	rewriting methods for operads. We introduce and use two new general notions: an almost composite product and an almost distributive law.

\subsection*{Acknowledgements} I am indebted to Yuri Ivanovich Manin for inspiring conversations about $F$-manifolds. Thanks are also due to Sergei Merkulov and Bruno Vallette for useful comments on a draft version of this paper.  

\subsection*{Conventions} This is a short note, and we do not intend to overload it with excessive recollections. For relevant information on symmetric operads and Koszul duality, we refer the reader to the monograph \cite{LV}, and for information on shuffle operads, Gr\"obner bases and rewriting systems to the monograph \cite{BrDo}.  

All operads in this paper are defined over a field $\k$ of characteristic zero. We assume all operads reduced ($\calP(0)=0$) and connected ($\calP(1)=\k$). The (co)augmentation (co)ideal of an (co)operad $\calP$ is denoted $\overline{\calP}$. The Koszul dual of a quadratic operad $\calP$ is denoted $\calP^{\ac}$. When writing down elements of operads, we use arguments $a_1$, \ldots, $a_n$ as placeholders, and capital Latin letters $X_1$, \ldots, $X_n$ as arguments (belonging to actual graded vector spaces on which operads act). Thus, the pre-Lie identity in the operadic form is
\begin{equation}\label{eq:pre-Lie-op}
(a_1\cdot a_2)\cdot a_3 - a_1\cdot (a_2\cdot a_3) = (a_1\cdot a_3)\cdot a_2- a_1\cdot (a_3\cdot a_2) ,
\end{equation}
and the signs in \eqref{eq:pre-Lie} arise from applying this to a decomposable tensor $X_1\otimes X_2\otimes X_3$ and using the standard Koszul sign rule. 

\section*{The Lie filtration and $F$-manifold algebras} 

We begin with making the statements from the introduction completely clear. Let us begin with a precise operadic definition of the protagonist of this paper.

\begin{definition}
The operad $\FM$ of \emph{$F$-manifold algebras} is generated by a symmetric binary operation $-\circ-$ and a skew-symmetric binary operation $[-,-]$ satisfying the associativity relation and the Jacobi identity 
\begin{gather*}
(a_1\circ a_2)\circ a_3=a_1\circ(a_2\circ a_3),\\
[[a_1,a_2], a_3]+[[a_2,a_3], a_1]+[[a_3,a_1],a_2]=0,
\end{gather*}
and related to each other by the Hertling--Manin relation
\begin{multline*}
[a_1\circ a_2,a_3\circ a_4]=
[a_1\circ a_2, a_3]\circ a_4+[a_1\circ a_2, a_4]\circ a_3+a_1\circ [a_2, a_3\circ a_4]+a_2\circ [a_1, a_3\circ a_4]-\\
(a_1\circ a_3)\circ[a_2,a_4]-(a_2\circ a_3)\circ[a_1,a_4]-(a_2\circ a_4)\circ[a_1,a_3]-(a_1\circ a_4)\circ[a_2,a_3] .
\end{multline*}
\end{definition}

Our next step is to define the filtration of the pre-Lie operad which we shall need.

\begin{definition}
The Lie filtration $F^\bullet\PL$ of the operad of pre-Lie algebras is defined as the filtration by powers of the ideal generated by the Lie bracket $[-,-]$. In other words, $F^k\PL$ is the span of tree tensors in which at least $k$ vertices are labelled with~$[-,-]$.
\end{definition}

Now everything is prepared for our main result to be stated.

\begin{theorem}\label{th:main}
The operad $\FM$ is the associated graded object for the Lie filtration: we have an operad isomorphism 
 \[
\gr_{F}\PL\cong\FM . 
 \]
In particular, \[\dim\FM(n)=n^{n-1},\] 
and the $S_n$-module $\FM(n)$ is isomorphic to the module of rooted trees on $\{1,\ldots,n\}$.
\end{theorem}

Note that the same filtration can be defined for the operad $\Ass$ of associative algebras. In that case, it is well known~\cite{LL,MaRe} that $\gr_F\Ass\cong\Poisson$. To prove it, one may first note that a direct computation demonstrates that relations of $\Poisson$ hold in $\gr_F\Ass$, which leads to a surjective map between these operads. Establishing that this map is an isomorphism can then be done by proving that this operads are of the same size, which requires some extra work. It turns out that a similar strategy, albeit more involved in the ``extra work'' part, is available in our case. We shall make the first step and establish a lower bound on the operad $\FM$ (\emph{a posteriori}, this bound will turn out to be sharp), and then develop a general formalism of almost composite products and almost distributive laws needed to complete the proof (note that in the case of the Poisson operad, we are dealing with actual composite products and distributive laws, so the parallelism of proofs is very apparent).

\begin{lemma}\label{lm:LowerBound}
We have a surjection of operads 
 \[
\FM\twoheadrightarrow\gr_{F}\PL . 
 \]
\end{lemma}

\begin{proof}
Recall that the symmetrised pre-Lie product $-\circ-$ and the bracket $[-,-]$ in the operad $\PL$ satisfy the following relation:
\begin{multline*}
(a_1\circ a_2)\circ a_3 - a_1\circ (a_2 \circ a_3) - a_1\circ [a_2, a_3] - [a_1, a_2]\circ a_3 - 2[a_1, a_3]\circ a_2+\\
 [a_1, a_2\circ a_3] + [a_1\circ a_2, a_3] + [[a_1, a_3], a_2] = 0 ,
\end{multline*}
see  \cite[Example~5.6.4.1]{BrDo}. Considering it as a relation in a shuffle operad  and computing its S-polynomial with itself, one arrives at a relation 
\begin{multline*}
 -[a_1\circ a_2, a_3]\circ a_4  - [a_1\circ a_2, a_4] \circ a_3  +  [a_1, a_4]\circ (a_2\circ a_3)  +  [a_1, a_3]\circ (a_2\circ a_4)  - \\
 [a_1, a_3\circ a_4]\circ a_2  +  [a_1\circ a_2, a_3\circ a_4]  +  a_1\circ ([a_2, a_3]\circ a_4)  +  a_1\circ ([a_2, a_4]\circ a_3)  -  \\
 a_1\circ [a_2, a_3\circ a_4]  +  [[a_1, a_3], a_2]\circ a_4  +  [[a_1, a_4], a_2]\circ a_3  +  2 [[a_1, a_4], a_3]\circ a_2  +\\
  [a_1, [a_2, a_3]]\circ a_4  +  [a_1, [a_2, a_4]]\circ a_3  +  [a_1, a_4]\circ [a_2, a_3]  +  [a_1, a_3]\circ [a_2, a_4]  + \\
 [a_1, [a_3, a_4]]\circ a_2  -  [[a_1, a_4], a_2\circ a_3]  -  [[a_1, a_3], a_2\circ a_4]  -  [a_1, [a_2, a_3]\circ a_4]  - \\
 [a_1, [a_2 a_4]\circ a_3]  -  2 [[[a_1, a_4], a_3], a_2]  -  [[a_1, a_4], [a_2, a_3]]  -  [[a_1, a_3], [a_2, a_4]]  -  [[a_1, [a_3, a_4]], a_2]=0.
\end{multline*}
Let us consider the images of both of these relations in $\gr_{F}\PL$. The first one is a combination of terms where the bracket is used $0$, $1$, and $2$ times. In the associated graded object, this relation becomes the associativity of the product~$-\circ-$. The second one is a combination of terms where the bracket is used $1$, $2$, and $3$ times. In the associated graded object, this relation becomes  
\begin{multline*}
 -[a_1\circ a_2, a_3]\circ a_4  - [a_1\circ a_2, a_4] \circ a_3  +  [a_1, a_4]\circ (a_2\circ a_3)  +  [a_1, a_3]\circ (a_2\circ a_4)  -  [a_1, a_3\circ a_4]\circ a_2  +\\
  [a_1\circ a_2, a_3\circ a_4]  +  a_1\circ ([a_2, a_3]\circ a_4)  +  a_1\circ ([a_2, a_4]\circ a_3)  -  a_1\circ [a_2, a_3\circ a_4]  =0 , 
\end{multline*}
which, modulo associativity, is equivalent to the Hertling--Manin condition. It remains to notice that by a standard polarisation argument \cite{MaRe}, the operations $-\circ-$ and $[-,-]$ generate the operad $\PL$, so the operad $\gr_F\PL$ is generated by their cosets, and hence it is a homomorphic image of the operad $\FM$.
\end{proof}

\section*{Almost composite products and almost distributive laws} 

In the case of the operad of Poisson algebras, the relation between $-\circ-$ and $[-,-]$ is a rewriting rule allowing to get rid of all occurrences of products inside brackets, showing that there is a surjective map onto the Poisson operad from the composite product $\Com\circ\Lie$. In the case of $F$-manifold algebras, one can mimic this approach. For that, we introduce and study a new general operadic construction. 

\begin{definition}
Suppose that $\calP=\calT(\calX)/(\calR)$ and $\calQ=\calT(\calY)/(\calS)$ are two operads. The \emph{almost composite product of $\calP$ and $\calQ$}, denoted 
$\calP\triangledown_0\calQ$, is defined as
 \[
\calP\triangledown_0\calQ := \calP\bigvee\calQ / (\alpha(\beta_1,\ldots,\beta_k) \colon \alpha\in\calY, \beta_i\in\calX) .
 \]
\end{definition}

\begin{remark}
The usual composite product $\calP\circ\calQ$ of the underlying $\mathbb{S}$-modules of $\calP$ and $\calQ$ is the underlying $\mathbb{S}$-module of the operad
 \[
\calP\vee_0\calQ:=\calP\bigvee\calQ / (\alpha(\beta,\id,\ldots,\id) \colon \alpha\in\calY, \beta\in\calX) ,
 \]
whose relations are known as the trivial distributive law between $\calP$ and $\calQ$, see \cite[Sec.~8.6.4]{LV}, which explains our terminology. Note that unlike the trivial distributive laws, almost composite products are not defined by quadratic relations, unless $\calQ$ is generated by unary operations. This takes these operads outside the scope of commonly used methods of operad theory.   
\end{remark}

We shall now establish two general results about almost composite products, which are analogous to corresponding results about trivial distributive laws. 
First, we shall show that the almost composite product of two operads $\calP$ and $\calQ$ is an upper bound on a certain class of quotients of $\calP\bigvee\calQ$, in the same way as the trivial distributive law is an upper bound for general rewriting rules used to define distributive laws between operads~\cite{LV,Markl94}. 

\begin{proposition}\label{lm:ZeroBound}
Let $\calP=\calT(\calX)/(\calR)$ and $\calQ=\calT(\calY)/(\calS)$ be two operads, and let $\calO$ be an operad of the form
 \[
\calO := \calP\bigvee\calQ / (\alpha(\beta_1,\ldots,\beta_k)-f(\alpha(\beta_1,\ldots,\beta_k)) \colon \alpha\in\calY, \beta_i\in\calX) ,
 \]
where $f\colon\calY\circ\calX\to \calP\bigvee\calQ$ is a linear map whose image is contained in the right ideal generated by~$\calX$. 
Then for each $n\ge 1$, there is a surjection of vector spaces
 \[
(\calP\triangledown_0\calQ)(n)\twoheadrightarrow\calO(n) . 
 \]
\end{proposition}

\begin{proof}
Let us view the shuffle operads associated to $\calP\triangledown_0\calQ$ and $\calO$ as quotients of the free shuffle operad generated by $\calX\oplus\calY$. We fix some admissible order $\prec_0$ of shuffle tree monomials.  

We first note that from the algorithm of computing the reduced Gr\"obner basis for a given operad, it is clear that the reduced Gr\"obner basis for the operad~$\calP\triangledown_0\calQ$ consists of linear combinations of monomials where vertices labelled by elements of $\calY$ are ``closer to the leaves'', that is there is no vertex labelled by an element from $\calY$ has a child labelled by an element of $\calX$. 

Let us now examine the operad $\calO$. The situation with this operad is more complicated: generally, there is no choice of an admissible order for which $\alpha(\beta_1,\ldots,\beta_k)$ is the leading term of $\alpha(\beta_1,\ldots,\beta_k)-f(\alpha(\beta_1,\ldots,\beta_k))$. For that reason, we need to invoke a rewriting system argument. Let us define the following order on shuffle tree monomials: $T\prec T'$ if 
\begin{itemize}
\item[-] $n(T)>n(T')$, where $n(T)$ is the number of pairs of internal vertices $(v,v')$ of $T$, where $v$ is on the way from the root of $T$ to $v'$, the label of $v$ is  from $\calX$, and the label of $v'$ is from $\calY$;
\item[-] $n(T)=n(T')$, and $T\prec_0 T'$. 
\end{itemize}
This order is not a monomial order. However, it is a total order, so it certainly allows to convert relations into rewriting rules (so every set of relations gives rise to a rewriting system), and it is a well-order (so every rewriting system is convergent). Note that for the operad $\calP\triangledown_0\calQ$ only tree monomials $T$ with $n(T)=0$ appear in the process of Knuth--Bendix completion of the corresponding rewriting system \cite{KB}, so the direction of all relations is dictated by $\prec_0$, and the result of the completion is the reduced Gr\"obner basis. 

Let us now apply the Knuth--Bendix completion procedure to relations of the operad $\calO$. It is clear than when we consider the critical pairs coming from the mixed $\calP$-$\calQ$-relations and the $\calQ$-relations, the computation mimics the one we performed for the operad  $\calP\triangledown_0\calQ$, adding some extra terms for which the parameter $n(T)$ is higher; thus these terms are smaller with respect to the order $\prec$. This means that whenever the Knuth--Bendix procedure produced a new rewriting rule for $\calP\triangledown_0\calQ$, it produces a new rewriting rule for $\calO$, and the left hand side of that rule is the same. (In principle, completely new rewriting rules may arise here: if a critical pair does not produce a new rewriting rule for $\calP\triangledown_0\calQ$, it still may produce one for $\calO$.) By contrast, the critical pairs coming from the mixed $\calP$-$\calQ$-relations and the $\calP$-relations contribute nothing in the case of the operad~$\calP\triangledown_0\calQ$, but may result in new Gr\"obner basis elements for $\calO$. 

We observe that the left hands sides of the rewriting rule set of normal forms for the component~$\calO$ include all the left hand sides of the rewriting rule set of normal forms for the component~$\calP\bigvee\calQ$, so the  set of normal forms for the operad $\calO$ is a subset of the set of the normal forms for $(\calP\triangledown_0\calQ)(n)$, which proves our claim. 
\end{proof}

The analogy with distributive laws mentioned above suggests the following definition.

\begin{definition}
Let $\calP=\calT(\calX)/(\calR)$ and $\calQ=\calT(\calY)/(\calS)$ be two operads, and let $\calO$ be an operad of the form
 \[
\calO := \calP\bigvee\calQ / (\alpha(\beta_1,\ldots,\beta_k)-f(\alpha(\beta_1,\ldots,\beta_k)) \colon \alpha\in\calY, \beta_i\in\calX) ,
 \]
where $f\colon\calY\circ\calX\to \calP\bigvee\calQ$ is a linear map whose image is contained in the right ideal generated by~$\calX$. 
The map~$f$ is said to be an \emph{almost distributive law between $\calP$ and $\calQ$} if for each $n$ we have an isomorphism of 
$\mathbb{S}$-modules
 \[
(\calP\triangledown_0\calQ)(n)\cong\calO(n) . 
 \]
\end{definition}

\smallskip 

Let us now determine the minimal model of the almost composite product of two Koszul operads. This result is inspired by both the computation of the minimal model of the trivial distributive law between two Koszul operads \cite[Prop.~8.6.3]{LV} and the approach to infinity-morphisms of strong homotopy algebras using ``homotopy Koszul operads'' \cite{MV}.

\begin{proposition}\label{lm:PQresol}
Let $\calP=\calT(\calX)/(\calR)$ and $\calQ=\calT(\calY)/(\calS)$ be two Koszul operads. Consider the endomorphism $d$ of the free operad 
$\calT(s^{-1}\overline{\calQ^{\ac}}\oplus\overline{\calQ^{\ac}}\oplus s^{-1}\overline{\calP^{\ac}})$ 
defined on the generators as follows:
\begin{itemize}
\item[-] on the first group of generators, $d(s^{-1}\lambda)=d_{\Omega(\calQ^{\ac})}(s^{-1}\lambda)$,
\item[-] on the second group of generators, $d(\lambda)=s^{-1}\lambda - (\id\otimes s)d_{\Omega(\calQ^{\ac})}(s^{-1}\lambda)$,
\item[-] on the third group of generators, $d(s^{-1}\mu)=d_{\Omega(\calP^{\ac})}(s^{-1}\mu)$.
\end{itemize}
Then $d^2=0$. Moreover, the suboperad of $\calT(s^{-1}\overline{\calQ^{\ac}}\oplus\overline{\calQ^{\ac}}\oplus s^{-1}\overline{\calP^{\ac}})$ generated by  
\[s^{-1}\overline{\calQ^{\ac}}\oplus s^{-1}\overline{\calP^{\ac}}\oplus (\overline{\calQ^{\ac}}\circ (s^{-1}\overline{\calP^{\ac}}))\] is a free operad and a $d$-invariant subspace, and the quasi-free operad
 \[
(\calT(s^{-1}\overline{\calQ^{\ac}}\oplus s^{-1}\overline{\calP^{\ac}}\oplus (\overline{\calQ^{\ac}}\circ (s^{-1}\overline{\calP^{\ac}}))), d)
 \]
is a minimal model of the operad $\calP\triangledown_0\calQ$.
\end{proposition}

\begin{proof}
A direct inspection confirms that $d^2=0$ (note that this is only needed to check on the second group of generators, as on the two other groups it follows from the fact that the cobar complex of a cooperad is a chain complex). It is also immediate to check that the suboperad $s^{-1}\overline{\calQ^{\ac}}\oplus s^{-1}\overline{\calP^{\ac}}\oplus (\overline{\calQ^{\ac}}\circ (s^{-1}\overline{\calP^{\ac}}))$ is $d$-invariant; once again, it is clear that for the generators of the first and the second group, their images under $d$ are made of elements of the same kind, so only the generators of $\overline{\calQ^{\ac}}\circ (s^{-1}\overline{\calP^{\ac}})$ need to be inspected. Also, this suboperad is free, since every element can be uniquely represented as a composite of generators; for that, it is useful to note that $\overline{\calQ^{\ac}}$ (without $s^{-1}$) only appears in the third group of generators. 

Let us compute the homology of the operad
 \[
(\calT(s^{-1}\overline{\calQ^{\ac}}\oplus s^{-1}\overline{\calP^{\ac}}\oplus (\overline{\calQ^{\ac}}\circ (s^{-1}\overline{\calP^{\ac}}))), d) .
 \]
For that, we introduce a weight grading on this operad defined by assigning weight $0$ to generators from $s^{-1}\overline{\calP^{\ac}}$, and weight $1$ to generators from both $s^{-1}\overline{\calQ^{\ac}}$ and $\overline{\calQ^{\ac}}\circ (s^{-1}\overline{\calP^{\ac}})$. Note that the contributions of $d_{\Omega(\calP^{\ac})}$ to $d$ (those appear in the images under $d$ of generators from $s^{-1}\overline{\calP^{\ac}}$ and $\overline{\calQ^{\ac}}\circ (s^{-1}\overline{\calP^{\ac}})$) do not change weight, and all other contributions to $d$ increase weight by at least one. Thus, we are dealing with a filtered chain complex, and we may consider the corresponding spectral sequence. The differential of the first page of that spectral sequence kills all the higher homotopies for $\calP$; thus, homology of that differential is identified with the operad
 \[
\calT(\calX\oplus s^{-1}\overline{\calQ^{\ac}} \oplus(\overline{\calQ^{\ac}}\circ\calX))/(\calR) ,
 \]
with the obvious differential derived from $d$. To compute the homology of that differential, we introduce another filtration defined by assigning weight $0$ to generators from $\calX$, and weight $1$ to generators from both $s^{-1}\overline{\calQ^{\ac}}$ and $\overline{\calQ^{\ac}}\circ\calX$. Note that for each generator, the contribution of the map $\lambda\mapsto s^{-1}\lambda$ does not change weight, and the contributions of $d_{\Omega(\calQ^{\ac})}$ increase the weight by at least one. The differential of the first page of the spectral sequence of this new filtered complex has the homology 
 \[
\calT(\calX\oplus s^{-1}\overline{\calQ^{\ac}})/(\calR\oplus (s^{-1}\overline{\calQ^{\ac}}\circ\calX)) ,
 \]
with the differential $(0,d_{\Omega(\calQ^{\ac})})$. The homology of that differential is manifestly the operad $\calP\triangledown_0\calQ$, and there is no room for further differentials. 
\end{proof}

\section*{Main result}

\begin{proof}[Proof of Theorem~\ref{th:main}]
Let us first establish that $\dim(\Com\triangledown_0\Lie)(n)=\dim\PL(n)$ for all $n\ge 1$,  or in other words
 \[
f_{\Com\triangledown_0\Lie}(t)=\sum_{n\ge 1}\frac{\dim(\Com\triangledown_0\Lie)(n)}{n!}t^n=\sum_{n\ge 1}\frac{\dim\PL(n)}{n!}t^n =f_{\PL}(t).
 \]
It is well known that the right hand side is a solution to the functional equation 
\begin{equation}\label{eq:Tree}
f_{\PL}(t)=t\exp(f_{\PL}(t)) ,
\end{equation}
which follows either from the rooted trees construction of the operad $\PL$ as in \cite{CL}, or from the Koszul duality theory. Thus, it is sufficient to establish that the left hand side is a solution to the same functional equation. By Proposition \ref{lm:PQresol}, the minimal model of $\Com\triangledown_0\Lie$ is generated by 
 \[
s^{-1}\overline{\Com^{\ac}}\oplus s^{-1}\overline{\Lie^{\ac}}\oplus(\overline{\Lie^{\ac}}\circ (s^{-1}\overline{\Com^{\ac}})). 
 \]
The generating functions of the Euler characteristics of the corresponding operads are, respectively,
 \[
f_{\overline{\Com^{\ac}}}(t)=-(-\log(1+t)+t)=\log(1+t)-t \quad \text{and} \quad  f_{\overline{\Lie^{\ac}}}(t)=-(\exp(-t)-1+t)=1-t-\exp(-t),
 \]
so we have
\begin{gather*}
f_{s^{-1}\overline{\Lie^{\ac}}}(t)=\exp(-t)-1+t,
f_{s^{-1}\overline{\Com^{\ac}}}(t)=-\log(1+t)+t,\\
f_{\overline{\Lie^{\ac}}\circ (s^{-1}\overline{\Com^{\ac}})}(t)=1+\log(1+t)-t-\exp(\log(1+t)-t)=1+\log(1+t)-t-(1+t)\exp(-t),\\
\end{gather*}
and
\begin{multline*}
f_{s^{-1}\overline{\Com^{\ac}}}(t)+f_{s^{-1}\overline{\Lie^{\ac}}}(t)+f_{\overline{\Lie^{\ac}}\circ (s^{-1}\overline{\Com^{\ac}})}(t)=\\
-\log(1+t)+t+\exp(-t)-1+t+1+\log(1+t)-t-(1+t)\exp(-t)=\\
\exp(-t)+t-(1+t)\exp(-t)=t-t\exp(-t) .
\end{multline*}
Finally, it is known~\cite{MaRe1} that for a minimal model $(\calT(\calE),d)$ of any operad $\calP$, the series $t-f_{\calE}(t)$ is the compositional inverse of the series $f_{\calP}(t)$, so 
$t\exp(-t)$ is the compositional inverse of $f_{\Com\triangledown_0\Lie}(t)$, and
 \[
f_{\Com\triangledown_0\Lie}(t)\exp(-f_{\Com\triangledown_0\Lie}(t))=t ,
 \] 
which is the same as the functional equation \eqref{eq:Tree}, as required.

The result we just proved, together with Propositions \ref{lm:LowerBound} and \ref{lm:ZeroBound}, means that we have a  a diagram of finite-dimensional vector spaces
 \[
\PL(n)\cong(\Com\triangledown_0\Lie)(n)\twoheadrightarrow\FM(n)\twoheadrightarrow \gr_{F}\PL(n) , 
 \] 
where the first and the last term are of the same dimension, and all maps are isomorphisms and surjections, hence all maps must be isomorphisms. In particular, the maps $\FM(n)\twoheadrightarrow \gr_{F}\PL(n)$ arise from a map of operads, so they assemble into an operad isomorphism. The claims about the dimension and the $S_n$-action follow since they are known to hold for the operad $\PL$ \cite{CL}, and our filtration is equivariant.
\end{proof}

\section*{Concluding remarks}

\subsection*{Strong homotopy $F$-manifold algebras}

Our proof of the main result implies that the operad $\FM$ is obtained from operads $\calP$ and $\calQ$ by an almost distributive law. Using a perturbation argument similar to that in~\cite{DKres}, one can establish the following result.

\begin{proposition}
Suppose that $\calO$ is an operad obtained from Koszul operads $\calP$ and $\calQ$ by an almost distributive law, and suppose further that $\calO$ has finite-dimensional components. The shuffle operad $\calO$ admits a minimal resolution with generators 
 \[
s^{-1}\overline{\calQ^{\ac}}\oplus s^{-1}\overline{\calP^{\ac}}\oplus (\overline{\calQ^{\ac}}\circ (s^{-1}\overline{\calP^{\ac}})) .
 \]
\end{proposition}

Let us give a sketch of a proof. The only serious change in comparison to  \cite{DKres} is replacing Gr\"obner bases with rewriting systems, similar to how methods of Kobayashi~\cite{Kobayashi} extend those of Anick~\cite{Anick}. 

The first step, analogous to~\cite[Th.~2.2]{DKres}, which constructs a resolution for any shuffle operad with monomial relations, goes through unchanged. The second step, analogous to~\cite[Th.~4.1]{DKres},  amounts to a perturbation argument which allows one to incorporate lower terms of relations; this argument goes through for rewriting systems without any problem, since in fact it only requires knowing the leading term of each relations for computations, and the well-order property for termination of those computations. Since $\calO$ is obtained from $\calP$ and $\calQ$ by an almost distributive law, the proof of Proposition~\ref{lm:ZeroBound} (under our assumption on finite-dimensionality of components of the operad~$\calO$) shows that the operads $\calP\triangledown_0\calQ$ and~$\calO$ may be presented by rewriting systems with the same sets of left-hand sides, and hence the same associated monomial shuffle operad. 

By examining the perturbation argument in~\cite[Th.~4.1]{DKres} together with the proof of Proposition~\ref{lm:ZeroBound}, one sees that the differentials induced on the spaces of indecomposable elements of thus obtained resolutions of the operads $\calP\triangledown_0\calQ$ and $\calO$ have the same homology. Transferring the homotopy cooperad structure from the space of indecomposable elements to homology~\cite{DCV}, and recalling Proposition \ref{lm:PQresol}, one obtains the claimed result. 

\smallskip 

Thus, the minimal model of the operad $\FM$ has generators 
 \[
s^{-1}\overline{\Com^{\ac}}\oplus s^{-1}\overline{\Lie^{\ac}}\oplus(\overline{\Lie^{\ac}}\circ (s^{-1}\overline{\Com^{\ac}})) . 
 \]
That result is notable in the context of Merkulov's work~\cite{Merk03}, where the notion of an $F_\infty$-manifold is suggested. The operad controlling $F_\infty$-manifolds in the sense of~\cite{Merk03} is not cofibrant, since the Jacobi identity for the Lie bracket is suppose to hold strictly, not just up to homotopy (this is an inevitable consequence of thinking of the operation $[-,-]$ as of the bracket of vector fields even when some structures are relaxed up to homotopy). 

Our computation confirms that a cofibrant replacement of $\FM$ given by the minimal model is of the ``right shape''; its lowest level with respect to the hierarchy of $L_\infty$-operations recovers the definition from~\cite{Merk03}. This agreement of two results is particularly remarkable in the view of the fact that, unlike Merkulov, we only consider $F$-manifold algebras, and ignore the underlying geometry.

\subsection*{$F$-manifold algebras and pre-Lie commutative algebras}
As a final remark, let us note that there is another peculiar way of constructing $F$-manifold algebras from algebras over operads that are \emph{a priori} unrelated to geometry. Recall the following definition~\cite{Fo15,Mans14}.

\begin{definition}
The operad of pre-Lie commutative algebras is generated by a symmetric binary operation $-\circ-$ and a binary operation $-\cdot-$ without any symmetry satisfying the following relations:
\begin{gather*}
(a_1\circ a_2)\circ a_3=a_1\circ(a_2\circ a_3),\\
(a_1\cdot a_2)\cdot a_3 - a_1\cdot (a_2\cdot a_3) = (a_1\cdot a_3)\cdot a_2- a_1\cdot (a_3\cdot a_2) ,\\
(a_1\circ a_2)\cdot a_3=(a_1\cdot a_3)\circ a_2 + a_1\circ (a_2\cdot a_3).
\end{gather*}
\end{definition}

It turns out that every pre-Lie commutative algebra has a canonical structure of an $F$-manifold algebra.

\begin{proposition}
In any pre-Lie commutative algebra, the product $\circ$ and the bracket $[a_1,a_2]=a_1\cdot a_2-a_2\cdot a_1$ satisfy the $F$-manifold algebra identities. 
\end{proposition}

\begin{proof}
Associativity and the Jacobi identity are obvious, so we just need to verify the Hertling--Manin condition. Note that in a pre-Lie commutative algebra, we have
\begin{multline*}
P_{a_1}(a_2,a_3)=[a_1,a_2\circ a_3]-[a_1,a_2]\circ a_3-[a_1,a_3]\circ a_2=\\
a_1\cdot (a_2\circ a_3)-(a_2\circ a_3)\cdot a_1-(a_1\cdot a_2-a_2\cdot a_1)\circ a_3-(a_1\cdot a_3-a_3\cdot a_1)\circ a_2=\\
a_1\cdot (a_2\circ a_3)-(a_1\cdot a_2)\circ a_3-(a_1\cdot a_3)\circ a_2.
\end{multline*}
From this, the Hertling--Manin condition is quite easy to see:
\begin{multline*}
P_{a_1\circ a_2}(a_3,a_4)=(a_1\circ a_2)\cdot (a_3\circ a_4)-((a_1\circ a_2)\cdot a_3)\circ a_4-((a_1\circ a_2)\cdot a_4)\circ a_3=\\
(a_1\cdot (a_3\circ a_4))\circ a_2+(a_2\cdot (a_3\circ a_4))\circ a_1-((a_1\cdot a_3)\circ a_2 + a_1\circ (a_2\cdot a_3))\circ a_4-\\
((a_1\cdot a_4)\circ a_2 + a_1\circ (a_2\cdot a_4))\circ a_3=(a_1\cdot (a_3\circ a_4)-(a_1\cdot a_3)\circ a_4 -(a_1\cdot a_4)\circ a_3)\circ a_2+\\
(a_2\cdot (a_3\circ a_4)- (a_2\cdot a_3)\circ a_4 -(a_2\cdot a_4)\circ a_3)\circ a_1=P_{a_1}(a_3,a_4)\circ a_2+P_{a_2}(a_3,a_4)\circ a_1,
\end{multline*}
as required.
\end{proof}
We conjecture that this construction embeds the operad  $\FM$ into the operad of pre-Lie commutative algebras, which may lead to further ways of studying it.

\bibliographystyle{amsplain}
\providecommand{\bysame}{\leavevmode\hbox to3em{\hrulefill}\thinspace}

\end{document}